\documentclass[11pt]{amsart}

\usepackage[english]{babel}
\usepackage{amsmath, amsfonts, amssymb, amsthm, epsfig, graphicx, url, hyperref, color, tikz, cancel}

\newtheorem{theorem}{Theorem}[section]
\newtheorem{lemma}{Lemma}[section]
\newtheorem{corollary}{Corollary}[section]

\newtheorem{question}{Question}[section]
\newtheorem{remark}{Remark}[section]

\newcounter{theor}

\newtheorem{thm}[theor]{Theorem}

\DeclareMathOperator{\inter}{int}

\def\bd{\mathop\mathrm{bd}\nolimits}
\def\cl{\mathop\mathrm{cl}\nolimits}

\def\R{\mathbb{R}}
\def\N{\mathbb{N}}
\def\Z{\mathbb{Z}}
\def\Q{\mathbb{Q}}

\def\vol{\mathrm{vol}}

\def\G{\mathrm{G}}
\def\e{\mathrm{e}}
\newcommand{\dlat}{\mathrm{d}}

\def\Media#1#2#3#4{\mathcal{M}_{#1}^{#4}\left(#2,#3\right)}
\def\floor#1{\left\lfloor #1 \right\rfloor}
\def\ceil#1{\left\lceil #1 \right\rceil}
\def\G#1{\mathrm{G}_{#1}}
\def\fun{\Delta}
\def\symbol{\diamond}

\numberwithin{equation}{section}

\begin{document}
\title[Brunn-Minkowski inequalities for the lattice point enumerator]{Brunn-Minkowski type inequalities
for the lattice point enumerator}

\author{David Iglesias}
\address{Departamento de Matem\'aticas, Universidad de Murcia, Campus de
Espinar\-do, 30100-Murcia, Spain}
\email{david.iglesias@um.es} \email{jesus.yepes@um.es}

\author{Jes\'us Yepes Nicol\'as}

\author{Artem Zvavitch}
\address{Department of Mathematical Sciences, Kent State University, Kent, OH USA,}
\email{zvavitch@math.kent.edu}

\thanks{The first and second authors are
partially supported by MICINN/FEDER project PGC2018-097046-B-I00
and ``Programa de Ayudas a Grupos de Excelencia de la Regi\'on de Murcia'', Fundaci\'on S\'eneca,
19901/GERM/15.
The third author is supported in part by the United States - Israel Binational Science Foundation (BSF) and Simons Foundation.
This work has been partially developed during a research stay of
the first author at Kent State University, USA}



\begin{abstract}
Geometric and functional Brunn-Minkowski type inequa\-li\-ties for the lattice point enumerator $\G{n}(\cdot)$ are provided.
In particular, we show that $$\G{n}((1-\lambda)K + \lambda L + (-1,1)^n)^{1/n}\geq (1-\lambda)\G{n}(K)^{1/n}+\lambda\G{n}(L)^{1/n}$$ for any non-empty  bounded sets $K, L\subset\R^n$ and all $\lambda\in(0,1)$.

We also show that these new discrete versions imply the classical results, and  discuss some links with
other related inequalities.
\end{abstract}

\maketitle

\section{Introduction and notation}

As usual, we write $\R^n$ to represent the $n$-dimensional
Euclidean space, and we denote by $\e_i$ the $i$-th canonical unit
vector. We set $(x,y)$ for the open segment with endpoints $x,y\in\R^n$.
The $n$-dimensional volume of a compact set
$K\subset\R^n$, i.e., its $n$-dimensional Lebesgue measure, is
denoted by $\vol(K)$, and as a discrete counterpart, we use $|A|$
to represent the cardinality of a finite subset $A\subset\R^n$.

Let $\Z^n$ be the integer lattice, i.e., the lattice of all points with
integral coordinates in $\R^n$.
We will denote by $\G{n}(\cdot)$ the lattice point enumerator for the integer lattice $\Z^n$, i.e., $\G{n}(M)=|M\cap\Z^n|$.
By $\floor{x}$ we denote the floor function of the real number $x$, i.e., the
greatest integer less than or equal to $x$.
Similarly, $\ceil{x}$ represents the ceiling function of $x$, namely, the least integer greater than or equal to $x$.

Finally, given a set $M\subset\R^n$, $\chi_{_M}$ represents the characteristic function of $M$ and, moreover, we denote by $\inter M$, $\bd M$ and $\cl M$ its interior, boundary and closure, respectively.
Furthermore, given $r>0$, $rM$ stands for the set $\{rm:\, m\in M\}$.

\medskip

Relating the volume of the Minkowski addition of two sets in terms of their volumes, one is led to the famous \emph{Brunn-Minkowski inequality} (for extensive survey articles on this and related  inequalities we refer the
reader to \cite{Brt,G}).
One form of it asserts that if $\lambda\in(0,1)$ and $K$ and $L$ are non-empty compact subsets of $\R^n$ then
\begin{equation}\label{e:BM}
\vol\bigl((1-\lambda)K+\lambda L\bigr)^{1/n}\geq
(1-\lambda)\vol(K)^{1/n}+\lambda\vol(L)^{1/n}.
\end{equation}
Here $+$ is used for the Minkowski sum, i.e.,
$A+B=\{a+b:\, a\in A, \, b\in B\}$ for any non-empty sets $A, B\subset\R^n$.
Moreover, from the homogeneity of the volume, \eqref{e:BM} is equivalent to
\begin{equation}\label{e:BM_additive}
\vol\bigl(K+L\bigr)^{1/n}\geq
\vol(K)^{1/n}+\vol(L)^{1/n}.
\end{equation}

Next we move to the discrete setting, i.e., we consider finite
sets of (integer) points which are not necessarily
full-dimensional unless indicated otherwise. It can easily be seen
that one cannot expect to obtain a Brunn-Minkowski type inequality for
the cardinality in the classical form. Indeed, simply taking
$A=\{0\}$ to be the origin and any finite set $B\subset\Z^n$, we get
\[
|A+B|^{1/n}<|A|^{1/n}+|B|^{1/n}.
\]
Therefore, a discrete Brunn-Minkowski type inequality should either have a
different structure or involve modifications of the sets. A first example
is the simple inequality
\begin{equation}\label{e:|A+B|>|A|+|B|-1}
|A+B|\geq|A|+|B|-1,
\end{equation}
for finite $A,B\subset\Z^n$ (see e.g. \cite[Chapter~2]{Tao}).

In \cite{GG}, Gardner and Gronchi obtained an engaging
discrete Brunn-Min\-kows\-ki inequality: they proved that if $A,B$ are finite
subsets of the integer lattice $\Z^n$, with dimension $\dim B=n$, then
\begin{equation}\label{e:B-M_ineq_GG}
|A+B|\geq \bigl|D_{|A|}^B+D_{|B|}^B\bigr|.
\end{equation}
Here $D_{|A|}^B, D_{|B|}^B$ are {\it $B$-initial segments}: for $m\in\N$,
$D_m^B$ is the set of the first $m$ points of $\Z^n_+$ in the so-called
``$B$-order'', which is a particular order defined on $\Z^n_+$ depending
only on the cardinality of $B$. For a proper definition and a deep study
of it we refer the reader to~\cite{GG}. As consequences of
\eqref{e:B-M_ineq_GG} they also get two additional nice discrete
Brunn-Minkowski type inequalities:
\begin{equation*}
|A+B|^{1/n}\geq|A|^{1/n}+\dfrac{1}{(n!)^{1/n}}\bigl(|B|-n\bigr)^{1/n}
\end{equation*}
and, if $|B|\leq|A|$, then
\[
|A+B|\geq|A|+(n-1)|B|+\bigl(|A|-n\bigr)^{(n-1)/n}\bigl(|B|-n\bigr)^{1/n}-\dfrac{n(n-1)}{2}.
\]
These inequalities improve previous results obtained by Ruzsa in
\cite{Ru1,Ru2}.

\smallskip

An alternative to getting a ``classical'' Brunn-Minkowski type inequality
might be to transform (one of) the sets involved in the problem. In this regard,
in \cite{HCIYN} an extension $\bar{A}$ of $A$, is defined for any non-empty and finite set $A\subset\Z^n$
(for $n=1$ then $\bar{A}=A\cup\{\max(A)+1\}$, whereas for $n>1$ the set $\bar{A}$ is obtained by first adding the maximal cardinality section of $A$, and then applying the corresponding extension to every section of the latter
new set; we refer to Section 2.1 in \cite{HCIYN} for its precise definition and properties).
Using this technique the following discrete counterpart for \eqref{e:BM_additive} was shown:
\begin{thm}[\cite{HCIYN}]\label{t:B-M_discrete_adding}
Let $A,B\subset\Z^n$ be finite, $A,B\neq\emptyset$. Then
\begin{equation}\label{e:B-M_discrete_adding}
\bigl|\bar{A}+B\bigr|^{1/n}\geq |A|^{1/n}+|B|^{1/n}.
\end{equation}
The inequality is sharp.
\end{thm}

Aiming to get a discrete version of \eqref{e:BM}, it is worth noting the following:
for any pair of non-empty finite sets $A, B\subset\R^n$, by \eqref{e:|A+B|>|A|+|B|-1} (and using that $|A|, |B|\geq1$), one has
\begin{equation*}
\begin{split}
\bigl|(1-\lambda)A+\lambda B\bigr|&\geq \bigl|(1-\lambda)A\bigr|+\bigl|\lambda B\bigr|-1 =|A|+|B|-1 \\
&\geq (1-\lambda)|A|+\lambda |B|\geq \bigr((1-\lambda)|A|^{1/n}+\lambda |B|^{1/n}\bigr)^n,
\end{split}
\end{equation*}
where the last inequality follows from the convexity of the function $t\mapsto t^n$, for $t\geq0$.
Nevertheless this inequality is meaningless from a geometric point of view: the point is that while the quantities $|A|, |B|$ on the right-hand side are reduced by the factors $(1-\lambda)$ and $\lambda$, the sets $(1-\lambda)A$ and $\lambda B$ on the left-hand side have the same cardinality as $A$ and $B$, respectively. A possible solution would be to involve a more natural way to ``count points'' according to dilatations, namely, using the lattice point enumerator $\G{n}$ (for compact subsets of $\R^n$) instead of the cardinality $|\cdot|$ (for finite subsets of $\R^n$).

Again, one cannot expect to obtain a Brunn-Minkowski type inequality for $\G{n}$ in the classical form \eqref{e:BM} (which, as we have mentioned before, would be a similar situation to what happens for $|\cdot|$ regarding a discrete version of \eqref{e:BM_additive}). Indeed, just by taking $\lambda=1/2$, $K=[0,m-\epsilon]^n$ and $L=[0,m+\epsilon/2]^n$, with $m\in\N$ and $0<\epsilon<1$, one gets
\begin{equation}\label{e:counter_BM_G}
\G{n}\left(\frac{K+L}{2}\right)^{1/n}=m<m+\frac{1}{2}=\frac{\G{n}(K)^{1/n}+\G{n}(L)^{1/n}}{2}.
\end{equation}
Thus, as in \eqref{e:B-M_discrete_adding}, an alternative to get such an inequality for the lattice point enumerator would be to consider a certain extension of $(1-\lambda)K + \lambda L$. So, we pose the following question:
\begin{question}\label{q: BM_G}
Given compact sets $K, L\subset\R^n$ containing points from $\Z^n$,  what is the ``best'' way to define a set $M_\lambda$ such that
$(1-\lambda)K + \lambda L \subset M_\lambda$  and
\begin{equation*}
\G{n}(M_\lambda)^{1/n}\geq (1-\lambda)\G{n}(K)^{1/n}+\lambda\G{n}(L)^{1/n}
\end{equation*}
holds for a given $\lambda\in(0,1)$?
\end{question}

\section{Main results}
Here we give an answer to Question \ref{q: BM_G}, which we view as a discrete counterpart of \eqref{e:BM}:	
\begin{theorem}\label{t: BM_lattice_point_no_G(K)G(L)>0}
Let $\lambda \in (0,1)$ and let $K, L \subset \R^n$ be non-empty bounded sets.
Then
\begin{equation}\label{e: BM_lattice_point_no_G(K)G(L)>0}
\G{n}\bigl((1-\lambda)K + \lambda L + (-1,1)^n\bigr)^{1/n}\geq (1-\lambda)\G{n}(K)^{1/n}+\lambda\G{n}(L)^{1/n}.
\end{equation}
The inequality is sharp.
\end{theorem}

We point out the following: when dealing with arbitrary non-empty subsets  $K, L\subset\R^n$ (i.e., not necessarily bounded), from \eqref{e: BM_lattice_point_no_G(K)G(L)>0} we immediately get that
\begin{equation*}
\begin{split}
\G{n}\bigl(M_\lambda + (-1,1)^n\bigr)^{1/n}
&\geq \sup_{m \in \N}\G{n}\bigl((1-\lambda)K_m + \lambda L_m + (-1,1)^n\bigr)^{1/n}\\
&\geq \sup_{m \in \N}\left((1-\lambda)\G{n}(K_m)^{1/n}+\lambda\G{n}(L_m)^{1/n}\right)\\
&=(1-\lambda)\G{n}(K)^{1/n}+\lambda\G{n}(L)^{1/n}
\end{split}
\end{equation*}
for any $\lambda \in (0,1)$, where $M_\lambda=(1-\lambda)K + \lambda L$, $K_m=K\cap [-m, m]^n$ and $L_m=L\cap [-m, m]^n$. So, for the sake of simplicity, we will present here the results in the setting of bounded sets, although they also hold in full generality.

The core ingredient in the proof of Theorem \ref{t: BM_lattice_point_no_G(K)G(L)>0} is its functional analogue.
To introduce such a result, we present an analytical counterpart for functions of the Brunn-Minkowski inequality,
the so-called \emph{Borell-Brascamp-Lieb inequality}, originally proved in
\cite{Borell} and \cite{BL}.
To this end, we first recall the definition of the $p$-mean $\Media{p}{\cdot}{\cdot}{\lambda}$
of two non-negative numbers, where $p$ is a parameter varying in $\R\cup\{\pm\infty\}$ (for a general reference for $p$-means of non-negative numbers, we refer the reader to the classic text of Hardy, Littlewood, and P\'olya \cite{HaLiPo} and to the handbook \cite{Bu}):
we consider first the case $p\in\R$, with $p\neq0$: given $a,b>0$ we set
\[
\Media{p}{a}{b}{\lambda}=\bigl((1-\lambda)a^p+\lambda b^p\bigr)^{1/p}.
\]
For $p=0$, we write $\Media{0}{a}{b}{\lambda}=a^{1-\lambda}b^{\lambda}$.
To complete the picture, for $p=\pm \infty$ we set $\Media{\infty}{a}{b}{\lambda}=\max\{a,b\}$ and
$\Media{-\infty}{x}{y}{\lambda}=\min\{a,b\}$. Finally, if $ab=0$, we define $\Media{p}{a}{b}{\lambda}=0$ for all $p\in\R\cup\{\pm\infty\}$. Note that $\Media{p}{a}{b}{\lambda}=0$, if $ab=0$, is redundant for all $p\leq0$, however it is relevant for $p>0$.
The reason to modify in this way the definition of $p$-mean given in \cite{HaLiPo} is due to the classical statement of the Borell-Brascamp-Lieb inequality, which is
collected below. In fact, without such a modification, the thesis of
the latter result would not have mathematical interest.

The following theorem (see
also \cite{G} for a detailed presentation), as previously stated, can be regarded as the
functional counterpart of the Brunn-Minkowski inequality.
In fact, a straightforward proof of \eqref{e:BM} for compact sets $K, L\subset\R^n$ of positive volume
can be obtained by applying \eqref{e:BBL_means} to the characteristic functions $f=\chi_{_K}$, $g=\chi_{_L}$ and
$h=\chi_{_{(1-\lambda)K + \lambda L}}$ with $p=\infty$. Moreover, in the literature, the case $p=0$ is referred to
as the \emph{Pr\'ekopa-Leindler inequality}, which is a powerful tool,
in particular when dealing with \emph{log-concave functions}.
\begin{thm}[The Borell-Brascamp-Lieb inequality]\label{t:BBL_means}
Let $\lambda\in(0,1)$. Let $-1/n\leq p\leq\infty$ and let $f,g,h:\R^n\longrightarrow\R_{\geq0}$ be measurable
functions such that
\begin{equation*}
h((1-\lambda)x + \lambda y)\geq\Media{p}{f(x)}{g(y)}{\lambda}
\end{equation*}
for all $x,y\in\R^n$. Then
\begin{equation}\label{e:BBL_means}
\int_{\R^n}h(x)\,\dlat x\geq
\Media{\frac{p}{np+1}}{\int_{\R^n}f(x)\,\dlat x}{\int_{\R^n}g(x)\,\dlat x}{\lambda}.
\end{equation}
\end{thm}

To state our main result and henceforth, we will need the following notation:
for a function $\phi:\R^n\longrightarrow\R_{\geq0}$, we denote by $\phi^{\symbol}:\R^n\longrightarrow\R_{\geq0}$ the function given by	
\[\phi^{\symbol}(z) = \sup_{u\in(-1,1)^n}\phi(z+u)  \quad \text{ for all } z\in\R^n.\]
Such an extension of the function $\phi$ is nothing but the \emph{Asplund product} $\star$
(also referred to as the \emph{sup-convolution}, which can be regarded as the functional analogue of the Minkowski sum in the setting of log-concave functions, for which we refer the reader to \cite[Section~9.5]{Sch2} and the references therein)
of the functions $\phi$ and $\chi_{(-1,1)^n}$. Indeed,
\begin{equation*}
\begin{split}
\phi^{\symbol}(z) &= \sup_{u\in(-1,1)^n}\phi(z+u) =\sup_{u \in \R^n} \phi (z+ u)  \chi_{(-1,1)^n}(-u)\\
&=\sup_{u_1+u_2=z} \phi (u_1)  \chi_{(-1,1)^n}(u_2) = \bigl(\phi \star \chi_{(-1,1)^n}\bigr)(z).
\end{split}
\end{equation*}

\smallskip

Our main result reads as follows:
\begin{theorem}\label{t: BBL discreta (caso p-media)}
Let $\lambda\in(0,1)$ and let $K, L \subset \R^n$ be non-empty bounded sets. Let $-1/n\leq p\leq\infty$ and let
$f,g,h:\R^n\longrightarrow \R_{\geq 0}$ be functions such that
\[h((1-\lambda)x + \lambda y)\geq\Media{p}{f(x)}{g(y)}{\lambda}\]
for all $x \in K$, $y \in L$. Then
\begin{equation}\label{e: BBL discreta (caso p-media)}
\sum_{z \in M\cap\Z^n} h^\symbol(z)\geq \Media{\frac{p}{np+1}}{\sum_{x \in K\cap\Z^n} f(x)}{\sum_{y \in L\cap\Z^n} g(y)}{\lambda},
\end{equation}
where $M = (1-\lambda)K+\lambda L + (-1,1)^n$.
\end{theorem}

Let $\mathcal{B}=\{v_1\dots,v_n\}$ be a basis of an
$n$-dimensional lattice $\Lambda\subset\R^n$.
Under the same assumptions of the above result, we may consider the auxiliary functions
$f_{\mathcal{B}}, g_{\mathcal{B}}, h_{\mathcal{B}}: \R^n \rightarrow \R_{\geq0}$ defined by
\[ f_{\mathcal{B}}(x) = f(\varphi(x)), \quad g_{\mathcal{B}}(x) = g(\varphi(x))\text{ and } \quad h_{\mathcal{B}}(x) = h(\varphi(x)),\]
for all $x\in \R^n$, where $\varphi:\R^n\rightarrow\R^n$ is the linear (bijective) map given by $\varphi(x)=\sum_{i = 1}^n x_i v_i$ for any $x=(x_1,\dots,x_n)\in\R^n$. Thus, as an immediate consequence of the previous result (applied to the functions $f_{\mathcal{B}}, g_{\mathcal{B}}$ and $h_{\mathcal{B}}$), we get that
\begin{equation*}\label{e: BBL discrete (lattices)}
\sum_{z \in M\cap\Lambda} h^{\symbol_{_\mathcal{B}}}(z) \geq \Media{\frac{p}{np+1}}{\sum_{x \in K\cap\Lambda} f(x)}{\sum_{y \in L\cap\Lambda} g(y)}{\lambda},
\end{equation*}
where $M = (1-\lambda)K+\lambda L + \varphi((-1,1)^n)$ and $h^{\symbol_{_\mathcal{B}}}(z) = \sup_{u\in \varphi((-1,1)^n)} h(z+u)$ for all $z\in\R^n$.
So, Theorem \ref{t: BBL discreta (caso p-media)} also holds in the setting of an $n$-dimensional lattice $\Lambda\subset\R^n$.	

\smallskip

We would like to point out that Theorem \ref{t:BBL_means} admits an equivalent version for $p$-sums (instead of $p$-means). In this regard, in \cite{IYN} a discrete version of the Borell-Brascamp-Lieb inequality for $p$-sums was provided. However, in contrast to the continuous setting, where one may directly obtain Theorem \ref{t:BBL_means} from the corresponding version for $p$-sums (and vice versa) because
of the homogeneity of the volume, one cannot expect to derive in a similar way a discrete version of Theorem \ref{t:BBL_means} (like Theorem \ref{t: BBL discreta (caso p-media)}) from \cite[Theorem~2.1]{IYN}.

\smallskip

In order to prove Theorem \ref{t: BBL discreta (caso p-media)}, we state an auxiliary result that will allow us to get the one-dimensional case of the above-mentioned Brunn-Minkowski type inequality for the lattice point enumerator (cf. Question \ref{q: BM_G}).
\begin{lemma}\label{l: G(M)+Z(M)>= (1-lam)G(K)+lamG(L)}
Let $\lambda \in (0,1)$ and let $K, L, M \subset \R$ be non-empty sets such that $(1-\lambda)K + \lambda L \subset M$. If $M = \bigcup_{i=1}^s [a_i, b_i]$ is a finite union of pairwise disjoint compact intervals then
\begin{equation}\label{e: G(M)+Z(M)>= (1-lam)G(K)+lamG(L)}
\G{1}(M) + \fun(M) \geq (1-\lambda) \G{1}(K) + \lambda \G{1}(L),
\end{equation}
where $\fun(M)$ denotes the number of non-integer extreme points of $M$, namely
\[ \fun(M) = | \{ a_i \notin \Z: 1\leq i \leq s\} | + | \{ b_i \notin \Z: 1\leq i \leq s \} |. \]
\end{lemma}	
\begin{proof}
We show the result by induction on the number of intervals $s$ of $M$. For the case $s=1$, i.e., when
$M = [a_1, b_1]$ is a (non-empty) compact interval, we have on the one hand that
$\G{1}(M) = \floor{b_1} - \ceil{a_1} +1$.
Moreover, denoting by $a = \inf K$, $b = \sup K$, $c = \inf L$ and $d = \sup L$, we clearly get
$\G{1}(K)\leq \G{1}([a,b]) = \floor{b} - \ceil{a} +1$ and $\G{1}(L) \leq \G{1}([c,d]) = \floor{d} - \ceil{c} +1$.
On the other hand, the inclusion $(1-\lambda)K + \lambda L \subset M$ implies that
\begin{equation*}
\floor{b_1} \geq b_1 - \chi_{_{\R\setminus\Z}}(b_1)
\geq (1-\lambda)\floor{b} + \lambda\floor{d} - \chi_{_{\R\setminus\Z}}(b_1)
\end{equation*}
and
\begin{equation*}
-\ceil{a_1} \geq -a_1 - \chi_{_{\R\setminus\Z}}(a_1)
\geq -(1-\lambda)\ceil{a} - \lambda \ceil{c} - \chi_{_{\R\setminus\Z}}(a_1),
\end{equation*}
and thus
\begin{equation*}
\floor{b_1} - \ceil{a_1} +1
\geq (1-\lambda)(\floor{b}-\ceil{a} + 1) + \lambda(\floor{d}-\ceil{c}+1)-\fun(M).
\end{equation*}
This, together with the above upper bounds for the lattice point enumerator $G_1$ of $K$ and $L$, yields
$\G{1}(M)\geq (1-\lambda)\G{1}(K) + \lambda \G{1}(L)-\fun(M)$, which shows the case $s=1$.

So, we suppose that the inequality is true for $s\geq1$ and assume that
$M = \bigcup_{i=1}^{s+1} [a_i, b_i]$, where $b_i < a_{i+1}$ for all $1\leq i\leq s$.

Denoting by $M_1 = [a_1, b_1]$ and $M_2 = \bigcup_{i=2}^{s+1} [a_i, b_i]$, we may assume, without loss of generality, that $M_1\cap ((1-\lambda)K + \lambda L)\neq\emptyset$ (otherwise, the result follows directly from the induction hypothesis applied to the sets $K$, $L$ and $M_2$). Hence, we may define $m = \sup \bigl(M_1\cap((1-\lambda)K + \lambda L)\bigr)$ and then, since $K$ and $L$ are bounded (because $(1-\lambda)K + \lambda L\subset M$), there exist $k \in \cl K$ and $l \in \cl L$ such that $(1-\lambda)k + \lambda l = m$.
Thus, considering the sets $K_1 = \{x \in K: x \leq k\}$, $K_2 = K \setminus K_1$, $L_1 = \{x \in L: x \leq l\}$ and $L_2 = L \setminus L_1$, we clearly have that $(1-\lambda)K_1 + \lambda L_1 \subset M_1$ and $(1-\lambda)K_2 + \lambda L_2 \subset M_2$. Therefore, applying the induction hypothesis (and taking into account that $M_1$ are $M_2$ are disjoint), we get
\begin{equation*}
\begin{split}
\G{1}(M) + \fun(M) &= \G{1}(M_1) + \G{1}(M_2) + \fun(M_1) + \fun(M_2) \\
&\geq (1-\lambda)\bigl(\G{1}(K_1)+\G{1}(K_2)\bigr) + \lambda\bigl(\G{1}(L_1)+\G{1}(L_2)\bigr)\\
&= (1-\lambda) \G{1}(K) + \lambda \G{1}(L),
\end{split}
\end{equation*}
as desired.
\end{proof}

\begin{remark}
One might think that if $(1-\lambda)K + \lambda L$ is itself a finite union of (pairwise disjoint) compact intervals, the set $M = (1-\lambda)K + \lambda L$ would yield a tighter inequality in
\eqref{e: G(M)+Z(M)>= (1-lam)G(K)+lamG(L)} than other greater sets (with respect to set inclusion). Nevertheless, this is not true in general: if we consider $K = [-2m, -1] \cup [1,2m]$ with $m \in \N$, $L = \{0\}$ and $\lambda = 1/2$, then for $M = (1-\lambda)K + \lambda L=[-m, -1/2] \cup [1/2, m]$ we get $\G{1}(M) + \fun(M) = 2m + 2$ whereas for $M' = [-m, m]$ we have
$\G{1}(M') + \fun(M') = 2m+1$.
\end{remark}

We notice that, as shown in \eqref{e:counter_BM_G}, the quantity $\fun(M)$ cannot be (in general) omitted.
However, we can rewrite \eqref{e: G(M)+Z(M)>= (1-lam)G(K)+lamG(L)} to provide an answer to Question \ref{q: BM_G} for $n=1$, which is the one-dimensional case of Theorem \ref{t: BM_lattice_point_no_G(K)G(L)>0}:
\begin{lemma}\label{l: 1-dim_BM_G}
Let $\lambda \in (0,1)$ and let $K, L \subset \R$ be non-empty bounded sets. Then
\begin{equation}\label{e: 1-dim_BM_G}
\G{1}\bigl((1-\lambda)K + \lambda L + (-1,1)\bigr) \geq (1-\lambda) \G{1}(K) + \lambda \G{1}(L).
\end{equation}
The inequality is sharp.
\end{lemma}
\begin{proof}
Let $M = \bigcup_{x \in(1-\lambda)K + \lambda L} [\floor{x}, \ceil{x}]$. Clearly, $M$ is a finite union of compact intervals (since $K$ and $L$ are bounded) containing $(1-\lambda)K + \lambda L$. From Lemma
\ref{l: G(M)+Z(M)>= (1-lam)G(K)+lamG(L)} we then obtain $\G{1}(M)+\fun(M)\geq (1-\lambda) \G{1}(K) + \lambda \G{1}(L)$, which, together with the facts that $\fun(M) = 0$ and
$M \cap \Z =\left((1-\lambda)K + \lambda L + (-1,1)\right) \cap \Z$, yields \eqref{e: 1-dim_BM_G}.
		
Finally, in order to show that equality may be attained, we consider $K = L = [0,m]$ with $m \in \N$, for which $\G{1}\bigl((1-\lambda)K + \lambda L+(-1,1)\bigr) = m+1 = (1-\lambda) \G{1}(K) + \lambda \G{1}(L)$ for all $\lambda\in(0,1)$.
\end{proof}

\begin{remark}\label{r:best_interval}
Since both sides of \eqref{e: 1-dim_BM_G} are invariant under translations by integers of the sets $K$ and $L$,
we may replace $(-1,1)$ (in \eqref{e: 1-dim_BM_G}) by any other interval $(m,m+2)$, with $m\in\Z$.

We note however that the solution to Question \ref{q: BM_G} provided by Lemma \ref{l: 1-dim_BM_G} (for $n=1$), via
$M=(1-\lambda)K + \lambda L+(-1,1)$, for all $\lambda\in(0,1)$, cannot be in general improved by means of any other interval strictly contained in $(-1,1)$.
Indeed, by considering $I=[-a,1)$, with $-1<-a<0$, and taking $K=[-1,0]$, $L=[-2,0]$ and
$\lambda\in(0,1)$ such that $\lambda+a<1$, we get that
\[
\begin{split}
\G{1}\bigl((1-\lambda)K + \lambda L + I\bigr)&=\G{1}\bigl([-1-\lambda-a,1)\bigr)=2\\
&<2+\lambda=(1-\lambda) \G{1}(K) + \lambda \G{1}(L).
\end{split}
\]
The case $I=(-1,a]$, for $0<a<1$, is completely analogous and thus, no interval smaller than $(-1,1)$ (with respect to set inclusion) can be taken into account.
\end{remark}

\medskip

Now we state some auxiliary results. The following lemma can be regarded as a discrete counterpart of the well-known \emph{Cavalieri Principle} (see \cite[Lemma~3.1]{IYN}).

\begin{lemma}[\cite{IYN}]\label{l: Fubini discreto}
Let $\Omega \subset \R^n$ be a finite set, let $f:\Omega\longrightarrow\R_{\geq 0}$ and suppose $f(\Omega)\subset\{k_0, k_1, \dots, k_r\}$ where $0=k_0<k_1<\dots<k_r$.
Then
\begin{equation*}\label{e: Fubini discreto}
\sum_{x\in\Omega} f(x) = \sum_{i=1}^r (k_i - k_{i-1})\Bigl|\{ x \in \Omega : f(x) \geq k_i \}\Bigr|
=\int_0^{\infty} \Bigl|\{ x \in \Omega : f(x) \geq t \}\Bigr| \, \dlat t.
\end{equation*}
\end{lemma}


\begin{corollary}\label{c: Fubini discreto (lattice point enumerator)}
Let $\Omega\subset\R^n$ be a bounded set, let $f:\R^n\longrightarrow\R_{\geq 0}$ and suppose $f(\Omega\cap\Z^n)\subset\{k_0, k_1, \dots, k_r\}$ where $0=k_0<k_1<\dots<k_r$.
Then
\begin{equation*}
\sum_{x\in\Omega\cap\Z^n} f(x) = \sum_{i=1}^r (k_i - k_{i-1})\G{n}\bigl(\{ x \in \Omega : f(x) \geq k_i \}\bigr).
\end{equation*}
\end{corollary}

\medskip

The following result yields the case $n = 1$ of Theorem \ref{t: BBL discreta (caso p-media)} and
will be used to derive \eqref{e: BBL discreta (caso p-media)}.
\begin{lemma}\label{l: BBL Paso1 (caso p-media) new}
Let $\lambda \in (0,1)$ and let $\Omega_1, \Omega_2 \subset \R$ be non-empty bounded sets. Let $-1\leq p\leq\infty$ and let $f,g,h:\R \longrightarrow \R_{\geq 0}$ be functions such that
\begin{equation*}
h((1-\lambda)x + \lambda y)\geq\Media{p}{f(x)}{g(y)}{\lambda}
\end{equation*}
for all $x\in\Omega_1$, $y\in\Omega_2$. Then
\begin{equation*}
\sum_{z \in \Omega \cap \Z}h^\symbol(z) \geq \Media{\frac{p}{p+1}}{\sum_{x \in \Omega_1 \cap \Z}f(x)}{\sum_{y \in \Omega_2\cap \Z}g(y)}{\lambda},
\end{equation*}
where $\Omega = (1-\lambda)\Omega_1 + \lambda \Omega_2 + (-1,1)$.
\end{lemma}	
\begin{proof}
Clearly, we may assume that $\sum_{x \in \Omega_1 \cap \Z}f(x), \sum_{y \in \Omega_2 \cap \Z}g(y)>0$.
We consider the functions $F, G, H, H^\symbol:\R\longrightarrow\R_{\geq0}$ given by
\[ F(x) = \frac{f(x)}{a}, \quad G(y) = \frac{g(y)}{b}, \quad H(z) = \frac{h(z)}{c_p}, \quad H^\symbol(z)=  \frac{h^\symbol(z)}{c_p},\]
where
\[ a = \max_{x \in \Omega_1\cap\Z} f(x)>0, \quad b=\max_{y \in \Omega_2\cap\Z} g(y)>0 \quad \text{and} \quad c_p= \Media{p}{a}{b}{\lambda}>0.\]
Then \[\max_{x \in \Omega_1\cap\Z} F(x) = \max_{y \in \Omega_2\cap\Z} G(y)=1.\]
		
First, we show that, for any $x \in \Omega_1$, $y \in \Omega_2$, we have
\begin{equation}\label{e: H((1-lambda)x + lambda y) >= min(F(x), G(y)) new}
H((1-\lambda)x + \lambda y) \geq \min(F(x), G(y)).
\end{equation}
To this aim, it is enough to consider $x \in \Omega_1$, $y \in \Omega_2$ with $f(x)g(y)>0$. If $p\neq0$ and $p\neq\infty$, writing $\theta = \lambda b^p / c_p^p \in (0,1)$, we get
\begin{equation*}
\begin{split}
h((1-\lambda)x + \lambda y) &\geq \left((1-\lambda)f(x)^p + \lambda g(y)^p \right)^{1/p}\\
&= c_p \left( \frac{(1-\lambda)a^pF(x)^p + \lambda b^p G(y)^p}{c_p^p} \right)^{1/p}\\
&= c_p \bigl( (1-\theta)F(x)^p + \theta G(y)^p \bigr)^{1/p}\\
&\geq c_p \min\{F(x), G(y)\}.
\end{split}
\end{equation*}

For $p=0$, we have
\begin{equation*}
h((1-\lambda)x + \lambda y) \geq f(x)^{1-\lambda}g(y)^\lambda
= c_0F(x)^{1-\lambda}G(y)^\lambda\geq c_0 \min\{F(x), G(y)\}.
\end{equation*}

For $p=\infty$, $h((1-\lambda)x + \lambda y) \geq \max\{f(x), g(y)\} \geq c_{\infty} \min\{F(x), G(y)\}$ clearly holds. Therefore, we have shown \eqref{e: H((1-lambda)x + lambda y) >= min(F(x), G(y)) new}.
		
Next, note that the definition of $F$ and $G$ implies that the level sets
\[ \{ x \in \Omega_1: F(x) \geq t \}, \quad \{ y \in \Omega_2 : G(y) \geq t \} \]
are non-empty for any $t \in [0,1]$. Moreover, writing $\Omega_{\lambda}=(1-\lambda)\Omega_1 + \lambda\Omega_2$,
we deduce from \eqref{e: H((1-lambda)x + lambda y) >= min(F(x), G(y)) new} that
\[ \{ z \in \Omega_{\lambda} : H(z) \geq t \} \supset (1-\lambda)\{ x \in \Omega_1 : F(x) \geq t \} + \lambda\{ y \in \Omega_2 : G(y) \geq t \}\]
and thus, by Lemma \ref{l: 1-dim_BM_G}, we have
\begin{equation}\label{e: Inequality of level sets (1-dim) new}
\begin{split}
&\G{1}\bigl(\{ z \in \Omega_{\lambda}: H(z) \geq t \} + (-1,1)\bigr)\\
\geq (1-\lambda)&\G{1}\bigl(\{ x \in \Omega_1 : F(x) \geq t \}\bigr)
+ \lambda\G{1}\bigl(\{ y \in \Omega_2 : G(y) \geq t \}\bigr)
\end{split}
\end{equation}
for all $t \in [0,1]$.
		
Note that, since $H^\symbol(z+u) \geq H\bigl((z+u)-u\bigr)=H(z)$ for all $u \in (-1,1)$, we also have
\begin{equation}\label{e: level_sets_Hsymbol}
\{ z\in \Omega : H^\symbol(z)\geq t \} \supset \{ z \in \Omega_{\lambda} : H(z) \geq t \} + (-1,1).
\end{equation}
		
Finally, set $\{k_0,k_1,\dots,k_r\} \supset F\bigl(\Omega_1\cap\Z\bigr) \cup G\bigl(\Omega_2\cap\Z\bigr)\cup H^\symbol\big(\Omega\cap\Z\big)$, with $0=k_0<k_1<\dots<k_r$ where, for some $s\in\{1,\dots,r\}$,
\[k_s = \max_{x \in \Omega_1\cap\Z} F(x) = \max_{y \in \Omega_2\cap\Z} G(y)=1.\]
Then, by Corollary \ref{c: Fubini discreto (lattice point enumerator)}, \eqref{e: Inequality of level sets (1-dim) new} and \eqref{e: level_sets_Hsymbol}, we get
\begin{equation*}
\begin{split}
\sum_{z \in \Omega\cap\Z} h^\symbol(z) &= \sum_{z \in \Omega\cap\Z} c_p H^\symbol(z)
= c_p \sum_{i = 1}^r (k_i - k_{i-1}) \G{1}\bigl(\{ z \in \Omega : H^\symbol(z) \geq k_i \}\bigr)\\
&\geq c_p \sum_{i = 1}^s (k_i - k_{i-1}) \G{1}\bigl(\{ z \in \Omega : H^\symbol(z) \geq k_i \}\bigr)\\
&\geq c_p \sum_{i = 1}^s (k_i - k_{i-1}) \Bigl( (1-\lambda)\G{1}\bigl(\{ x \in \Omega_1: F(x) \geq k_i \}\bigr) \\
& \hspace*{3.6cm} +  \lambda\G{1}\bigl(\{ y \in \Omega_2: G(y) \geq k_i \}\bigr) \Bigr)\\
&= c_p \left( (1-\lambda)\sum_{x \in \Omega_1\cap \Z} F(x) + \lambda\sum_{y \in \Omega_2\cap \Z} G(y) \right)\\
&= c_p \left( \frac{1-\lambda}{a}\,\sum_{x \in\Omega_1\cap \Z} f(x) +
\frac{\lambda}{b}\,\sum_{y \in \Omega_2\cap \Z} g(y) \right)\\
&\geq \Media{\frac{p}{p+1}}{\sum_{x \in \Omega_1\cap \Z} f(x)}{\sum_{y \in\Omega_2\cap \Z} g(y)}{\lambda}.
\end{split}
\end{equation*}
If $p\neq0$ and $p\neq\infty$, the last inequality follows from the reverse H\"older inequality (see e.g. \cite[Theorem~1, page~178]{Bu}),
\[ a_1b_1 + a_2b_2 \geq \left(a_1^{-p} + a_2^{-p}\right)^{-1/p} \left( b_1^q + b_2^q \right)^{1/q},\]
where $q = p/(p + 1)$ is the H\"older conjugate of $(-p) \leq 1$, just by taking $a_1 = ((1-\lambda)^{1/p} a)^{-1}$, $a_2 = (\lambda^{1/p} b)^{-1}$, $b_1 = (1-\lambda)^{1/q}\sum_{x \in \Omega_1\cap\Z} f(x)$ and $b_2 = \lambda^{1/q}\sum_{y \in \Omega_2\cap\Z} g(y)$.

The case $p = 0$ follows from the Arithmetic-Geometric mean inequality, whereas the case $p=\infty$ is immediate,
since there the $(p/(p+1))$-mean coincides with the $1$-mean.
\end{proof}

We would like to introduce a few additional notations which will be used to prove Theorem \ref{t: BBL discreta (caso p-media)}: we write $M(t)=\{x\in \R^{n-1}: (x,t)\in M\}$ for the $(n-1)$-dimensional section at height $t\in \R$ (in the direction of $\e_n$) whereas $\pi_n(M)$ denotes the orthogonal projection of $M$ onto $\R\e_n$ (regarded as a subset of $\R$), namely $\pi_n(M) = \{t \in \R: M(t)\neq\emptyset \}$.

\begin{proof}[Proof of Theorem \ref{t: BBL discreta (caso p-media)}]
If $n=1$, the result follows immediately from Lemma \ref{l: BBL Paso1 (caso p-media) new}.
Now suppose that $n>1$ and assume that the theorem holds for dimension $n-1$. Let $t_K \in \pi_n(K)$, $t_L\in\pi_n(L)$ and set, for the sake of brevity, $t_\lambda=(1-\lambda)t_K+ \lambda t_L$. Moreover, we denote by $C_n=(-1,1)^n$, $C_{n-1}=(-1,1)^{n-1}\times\{0\}$,  $M_{n-1}=(1-\lambda)K(t_K) + \lambda L(t_L) + (-1,1)^{n-1}$ and $M_\lambda=(1-\lambda)K+ \lambda L$ (thus, $M=M_\lambda + (-1,1)^n$). Consider the functions $f_1, g_1, h_1:\R^{n-1}\longrightarrow\R_{\geq 0}$ given by
\[ f_1(x) = f(x,t_K), \quad g_1(x) = g(x,t_L),\]
\[ h_1(x) = h(x, t_\lambda)\]
for any $x\in \R^{n-1}$.
Since for all $x \in K(t_K)$, $y\in L(t_L)$ we have
\begin{equation*}\begin{split}
h_1((1-\lambda)x+\lambda y) &= h((1-\lambda)x+\lambda y,(1-\lambda)t_K+ \lambda t_L)\\
&\geq \Media{p}{f(x,t_K)}{g(y,t_L)}{\lambda} = \Media{p}{f_1(x)}{g_1(y)}{\lambda},
\end{split}\end{equation*}
we may assert that
\begin{equation*}
\sum_{z \in M_{n-1}\cap\Z^{n-1}} h_1^\symbol(z) \geq
\Media{\frac{p}{(n-1)p+1}}{\sum_{x \in K(t_K)\cap\Z^{n-1}} f_1(x)}{\sum_{y \in L(t_L)\cap\Z^{n-1}} g_1(y)}{\lambda}.
\end{equation*}
This, together with the fact that
\[ ((1-\lambda)K+\lambda L)((1-\lambda)t_K + \lambda t_L) \supset (1-\lambda) K(t_K) + \lambda L(t_L),\]
and hence $(M_\lambda+C_{n-1})(t_\lambda)\supset M_{n-1}$, yields, in terms of $f$, $g$ and $h$,
\begin{equation}\label{e: suma de secciones n-1 dimensionales (caso p-medias) new}
\begin{split}
&\sum_{z \in ((M_\lambda+C_{n-1})(t_\lambda))\cap\Z^{n-1}} h^{\symbol\symbol}(z,t_\lambda)\\
&\geq \Media{\frac{p}{(n-1)p+1}}{\sum_{x \in K(t_K)\cap\Z^{n-1}} f(x,t_K)}{\sum_{y \in L(t_L)\cap\Z^{n-1}} g(y,t_L)}{\lambda},
\end{split}
\end{equation}
where $h^{\symbol\symbol}:\R^n\longrightarrow\R_{\geq 0}$ is the function given by $h^{\symbol\symbol}(z) = \sup_{v\in C_{n-1}} h(z+v)$, for which we have $h^{\symbol\symbol}(x,t_\lambda) = h_1^\symbol(x)$ for all $x\in\R^{n-1}$.
		
Now, let $f_2, g_2, h_2:\R \longrightarrow \R_{\geq 0}$ be the functions defined by
\[ f_2(t) = \sum_{x \in K(t)\cap\Z^{n-1}} f(x,t),
\;\; g_2(t) = \sum_{y \in L(t)\cap\Z^{n-1}} g(y,t),\,\text{ and }\]
\[h_2(t) = \sum_{z \in ((M_\lambda+C_{n-1})(t))\cap\Z^{n-1}} h^{\symbol\symbol}(z,t). \]	
Hence, \eqref{e: suma de secciones n-1 dimensionales (caso p-medias) new} yields, in terms of $f_2$, $g_2$ and $h_2$,
\[ h_2((1-\lambda)t_K + \lambda t_L) \geq \Media{\frac{p}{(n-1)p+1}}{f_2(t_K)}{g_2(t_L)}{\lambda} \]
for any $t_K\in\pi_n(K)$, $t_L\in\pi_n(L)$, and thus we may use Lemma \ref{l: BBL Paso1 (caso p-media) new} with the sets $\pi_n(K)$, $\pi_n(L)$ and the functions $f_2$, $g_2$ and $h_2$ to obtain
\[ \sum_{t \in \Omega\cap\Z} h_2^{\symbol}(t) \geq \Media{\frac{p}{np+1}}{\sum_{t_K \in \pi_n(K)\cap\Z} f_2(t_K)}{\sum_{t_L \in \pi_n(L)\cap\Z} g(t_L)}{\lambda}, \]
where $\Omega = (1-\lambda)\pi_n(K) + \lambda \pi_n(L) + (-1,1)$. 
In the following we prove that $\sum_{t\in\Omega\cap\Z} h_2^\symbol(t) \leq \sum_{z \in M\cap\Z^n} h^\symbol(z)$, and hence the above inequality together with the relations
\[ \sum_{t_K\in\pi_n(K)\cap\Z} f_2(t_K) =  \sum_{x \in K\cap\Z^n} f(x), \]
\[ \sum_{t_L\in\pi_n(L)\cap\Z} g_2(t_L) = \sum_{y\in L\cap\Z^n} g(y), \]
shows the result.
Indeed, from the fact that $(u,-w)\in C_{n}$ for any $(u,0)\in C_{n-1}$ and $w\in(-1,1)$, we have $(M_\lambda+C_{n-1})(t+w)\subset M(t)$ for all $w\in(-1,1)$ and thus we get
\begin{equation*}
\begin{split}
\sum_{t\in\Omega\cap\Z} h_2^\symbol(t) &= \sum_{t\in\Omega\cap\Z} \, \sup_{w\in(-1,1)} h_2(t+w)\\
&= \sum_{t\in\Omega\cap\Z} \, \sup_{w\in(-1,1)} \sum_{x \in ((M_\lambda+C_{n-1})(t+w))\cap\Z^{n-1}} h^{\symbol\symbol}(x, t+w)\\
&\leq \sum_{t\in\Omega\cap\Z} \; \sum_{x \in M(t)\cap\Z^{n-1}} \sup_{w\in(-1,1)} h^{\symbol\symbol}(x,t+w)\\
&= \sum_{t\in\Omega\cap\Z} \; \sum_{x \in M(t)\cap\Z^{n-1}} \sup_{w\in(-1,1)} \sup_{v \in (-1,1)^{n-1}} h(x+v, t+w)\\
&= \sum_{z \in M\cap\Z^n} \sup_{u\in (-1,1)^n} h(z+u) = \sum_{z \in M\cap\Z^n} h^\symbol(z),
\end{split}
\end{equation*}
as claimed. This finishes the proof.
\end{proof}

\begin{corollary}\label{c: BM_lattice_point}
Let $\lambda \in (0,1)$ and let $K, L \subset \R^n$ be non-empty bounded sets. Let $-1/n\leq p\leq\infty$.
Then
\begin{equation}\label{e: BM_lattice_point_general}
\G{n}\bigl((1-\lambda)K + \lambda L + (-1,1)^n\bigr)\geq \Media{\frac{p}{np+1}}{\G{n}(K)}{\G{n}(L)}{\lambda}.
\end{equation}
The inequality is sharp.
\end{corollary}
\begin{proof}
The result is an immediate consequence of Theorem \ref{t: BBL discreta (caso p-media)}, just by taking $f=\chi_{_K}$, $g=\chi_{_L}$ and $h=\chi_{_{(1-\lambda)K+\lambda L}}$, for which we clearly have that $h^\symbol=\chi_{_{(1-\lambda)K+\lambda L+(-1,1)^n}}$.

\smallskip

Now, in order to show that the equality may be attained, we consider $K = L = [0,m]^n$ with $m \in \N$,
for which
$\G{n}\bigl((1-\lambda)K + \lambda L+(-1,1)^n\bigr)=\G{n}(K)=\G{n}(L)=(m+1)^n$.
\end{proof}

We notice that \eqref{e: BM_lattice_point_general} for $p=\infty$ yields \eqref{e: BM_lattice_point_no_G(K)G(L)>0}
for bounded sets $K,L\subset\R^n$ with $\G{n}(K)\G{n}(L)>0$. So, to prove Theorem
\ref{t: BM_lattice_point_no_G(K)G(L)>0}, it is enough to deal with the case in which (only) one of the sets, say $L$, has no integer points. To this aim, first we show the following auxiliary result:
\begin{theorem}\label{t: BM_lattice_point_[-(q-1)/q,(q-1)/q]}
Let $K, L \subset \R^n$ be bounded sets such that $\G{n}(K)\G{n}(L)>0$
and let $\alpha=m/q$ and $\beta=p/q$ with $m,p,q\in\N$ so that $\alpha + \beta\leq1$. Then
\begin{equation}\label{e: BM_lattice_point_[-(q-1)/q,(q-1)/q]}
\G{n}\left(\alpha K+\beta L+\left[-\frac{q-1}{q},\frac{q-1}{q}\right]^n\right)^{1/n}
\geq \alpha\G{n}(K)^{1/n} + \beta\G{n}(L)^{1/n}.
\end{equation}
\end{theorem}
The proof of Theorem \ref{t: BM_lattice_point_[-(q-1)/q,(q-1)/q]}
is ideologically similar to the proof we have provided for Corollary \ref{c: BM_lattice_point}. So only the sketch of the proof is presented here. The main step is to prove an analogue of Theorem \ref{t: BBL discreta (caso p-media)} and apply it to $f=\chi_{_K}$, $g=\chi_{_L}$ and $h=\chi_{_{\alpha K+\beta L}}$.  The main new ingredient in the proof of such an analogue is the one dimensional case, for which we require an adaptation of Lemma \ref{l: 1-dim_BM_G}, which is done below in Lemma \ref{l: G(alpha K+beta L+[-(q-1)/q,(q-1)/q])}.


\begin{lemma}\label{l: G(alpha K+beta L+[-(q-1)/q,(q-1)/q])}
Let $K, L \subset \R$ be non-empty bounded sets and let $\alpha=m/q$ and $\beta=p/q$ with $m,p,q\in\N$ so that $\alpha + \beta\leq1$. Then
\begin{equation}\label{e: G(alpha K+beta L+[-(q-1)/q,(q-1)/q])}
\G{1}\left(\alpha K+\beta L+\left[-\frac{q-1}{q},\frac{q-1}{q}\right]\right) \geq \alpha\G{1}(K) + \beta\G{1}(L).
\end{equation}
\end{lemma}
\begin{proof}
First we notice that, for any $x,y\in\R$, we have
\begin{equation}\label{e:property_floor_(mx+py)/q}
\floor{\frac{mx}{q}+\frac{py}{q}+\frac{q-1}{q}}\geq\frac{m\floor{x}+p\floor{y}}{q}.
\end{equation}
Indeed, given $z\in\R$, from the fact that
\[\frac{z}{q}+\frac{q-1}{q}\geq\frac{\floor{z}}{q}+\frac{q-1}{q}=c+\frac{r}{q}+\frac{q-1}{q}
\]
for some $c,r\in\Z$ with $0\leq r\leq q-1$, we get that
\begin{equation*}\label{e:property_floor_x/q}
\floor{\frac{z}{q}+\frac{q-1}{q}}\geq\frac{\floor{z}}{q}
\end{equation*}
for any $z\in\R$. This now implies that
\begin{equation*}
\floor{\frac{mx}{q}+\frac{py}{q}+\frac{q-1}{q}}\geq\frac{\floor{mx+py}}{q}\geq\frac{m\floor{x}+p\floor{y}}{q},
\end{equation*}
which yields \eqref{e:property_floor_(mx+py)/q}.

\medskip

Next we show the following:
if $K, L, M \subset \R$ are non-empty sets with $\alpha K+\beta L\subset M$ and
such that $M + [-(q-1)/q,(q-1)/q] = \bigcup_{i=1}^s [a_i, b_i]$ is a finite union of (pairwise disjoint)
compact intervals then
\begin{equation}\label{e: G(M+[-(q-1)/q,(q-1)/q])>=}
\G{1}\bigl(M+[-(q-1)/q,(q-1)/q]\bigr) \geq \alpha\G{1}(K) + \beta\G{1}(L).
\end{equation}
We prove it by induction on the number of intervals $s$ of $M+[-(q-1)/q,(q-1)/q]$. For $s=1$, i.e., when $M +[-(q-1)/q,(q-1)/q]= [a_1, b_1]$ is a (non-empty) compact interval, we have on the one hand that
$$
\G{1}\bigl(M+[-(q-1)/q,(q-1)/q]\bigr) = \floor{b_1} - \ceil{a_1} +1.
$$
Moreover, denoting by $a = \inf K$, $b = \sup K$, $c = \inf L$ and $d = \sup L$, we clearly get
$\G{1}(K)\leq \G{1}([a,b]) = \floor{b} - \ceil{a} +1$ and $\G{1}(L) \leq \G{1}([c,d]) = \floor{d} - \ceil{c} +1$.
On the other hand, the inclusion $\alpha K + \beta L \subset M$ implies that
\begin{equation*}
a_1\leq\alpha a+\beta c-\frac{q-1}{q}\leq\alpha b+\beta d+\frac{q-1}{q}\leq b_1.
\end{equation*}
Altogether, and using \eqref{e:property_floor_(mx+py)/q} jointly with the facts that
$\ceil{x}=-\floor{-x}$ for any $x\in\R$, and $\alpha+\beta\leq1$, we obtain
\begin{equation*}
\begin{split}
\G{1}\left(M+\left[-\frac{q-1}{q},\frac{q-1}{q}\right]\right) &= \floor{b_1} -\ceil{a_1} + 1\\
&\geq  \floor{\alpha b+\beta d+\frac{q-1}{q}} - \ceil{\alpha a+\beta c-\frac{q-1}{q}} + 1\\
&\geq\alpha\floor{b}+\beta\floor{d}-\alpha\ceil{a}+\beta\ceil{c}+1\\
&\geq\alpha\G{1}([a,b])+\beta\G{1}([c,d])\\
&\geq\alpha\G{1}(K)+\beta\G{1}(L).
\end{split}
\end{equation*}

Thus, we suppose that \eqref{e: G(M+[-(q-1)/q,(q-1)/q])>=} is true for $s\geq1$ and assume that
$M+[-(q-1)/q,(q-1)/q] = \bigcup_{i=1}^{s+1} [a_i, b_i]$, where $b_i < a_{i+1}$ for all $1\leq i\leq s$.

Denoting by $M_1$ and $M_2$ the subsets of $M$ such that $M_1+[-(q-1)/q,(q-1)/q]= [a_1, b_1]$ and
$M_2+[-(q-1)/q,(q-1)/q] = \bigcup_{i=2}^{s+1} [a_i, b_i]$, we may assume, without loss of generality, that $M_1\cap (\alpha K + \beta L)\neq\emptyset$ (otherwise, the result follows from the induction hypothesis applied
to the sets $K$, $L$ and $M_2$). Hence, we may define $m = \sup \bigl(M_1\cap(\alpha K + \beta L)\bigr)$ and then, since $K$ and $L$ are bounded, there exist $k \in \cl K$ and $l \in \cl L$ such that $\alpha k +  \beta l = m$.
Thus, considering the sets $K_1 = \{x \in K: x \leq k\}$, $K_2 = K \setminus K_1$, $L_1 = \{x \in L: x \leq l\}$ and $L_2 = L \setminus L_1$, and taking into account that $m+[-(q-1)/q,(q-1)/q]\subset[a_1, b_1]$, we clearly have that $\alpha K_1 + \beta L_1 \subset M_1$ and $\alpha K_2 + \beta L_2 \subset M_2$. Therefore, applying the induction hypothesis (and taking into account that $M_1+[-(q-1)/q,(q-1)/q]$ and $M_2+[-(q-1)/q,(q-1)/q]$ are disjoint), we get
\begin{equation*}
\begin{split}
&\,\G{1}\bigl(M+[-(q-1)/q,(q-1)/q]\bigr) \\
&= \G{1}\bigl(M_1+[-(q-1)/q,(q-1)/q]\bigr) + \G{1}\bigl(M_2+[-(q-1)/q,(q-1)/q]\bigr)\\
&\geq \alpha\G{1}(K_1)+\beta\G{1}(L_1) + \alpha\G{1}(K_2)+\beta\G{1}(L_2)
= \alpha\G{1}(K) + \beta\G{1}(L),
\end{split}
\end{equation*}
which shows \eqref{e: G(M+[-(q-1)/q,(q-1)/q])>=}.

\smallskip

Next we prove \eqref{e: G(alpha K+beta L+[-(q-1)/q,(q-1)/q])}. We observe that we may assume, without loss of generality, that $K$ and $L$ are compact. Indeed, otherwise, considering the compact sets $K'= K\cap\Z$ and $L'=L\cap\Z$, for which we have $\G{1}(K')=\G{1}(K)$, $\G{1}(L')=\G{1}(L)$ and, from the monotonicity of $\G{1}(\cdot)$,
$\G{1}\bigl(\alpha K+\beta L+[-(q-1)/q,(q-1)/q]\bigr)\geq\G{1}\bigl(\alpha K'+\beta L'+[-(q-1)/q,(q-1)/q]\bigr)$, we would get the result.
So, \eqref{e: G(alpha K+beta L+[-(q-1)/q,(q-1)/q])} follows from applying \eqref{e: G(M+[-(q-1)/q,(q-1)/q])>=} with $M=\alpha K+\beta L$, because the fact that $M$ is compact implies that there exists a finite sequence $m_1,\dots,m_r$ such that $\{m_1,\dots,m_r\}+[-(q-1)/q,(q-1)/q]=M+[-(q-1)/q,(q-1)/q]$. This concludes the proof.
\end{proof}

We are now in a position to prove Theorem \ref{t: BM_lattice_point_no_G(K)G(L)>0}.

\begin{proof}[Proof of Theorem \ref{t: BM_lattice_point_no_G(K)G(L)>0}]
For the sake of brevity, we will again denote by $M_\lambda=(1-\lambda)K+\lambda L$.
By Corollary \ref{c: BM_lattice_point} for $p=\infty$, and the monotonicity of $\G{n}(\cdot)$, it is enough to
show the result in the case in which $K=K\cap\Z^n$ (for which, clearly, $\G{n}(K)>0$) and $L=\{x\}$ with $x=(x_1,\dots,x_n)\notin\Z^n$.

First we show the case in which $\lambda=p/q$, $p,q\in\N$, is a rational number. Then,
writing $px=z+y$, with $z\in\Z^n$ and $y=(y_1,\dots,y_n)\in[0,1)^n$, and using
Theorem \ref{t: BM_lattice_point_[-(q-1)/q,(q-1)/q]}, we have
\begin{equation*}
\begin{split}
\G{n}\bigl(M_\lambda +(-1,1)^n\bigr)^{1/n}
&=\G{n}\left(\frac{q-p}{q}K+\frac{1}{q}z +\prod_{i=1}^n\left(-1+\frac{y_i}{q},1+\frac{y_i}{q}\right)\right)^{1/n}\\
&\geq\G{n}\left(\frac{q-p}{q}K+\frac{1}{q}z +\left[-\frac{q-1}{q},\frac{q-1}{q}\right]^n\right)^{1/n}\\
&\geq (1-\lambda)\G{n}(K)^{1/n}+\frac{1}{q}>(1-\lambda)\G{n}(K)^{1/n},
\end{split}
\end{equation*}
as desired.

Next we prove the case of an irrational $\lambda\in\R\setminus\Q$.
Let $I$ be the (possibly empty) subset of $\{1,\dots,n\}$ defined in the following way: $i\in I$ if and only if $x_i=a_i+b_i/\lambda$ for some $a_i,b_i\in\Z$. We point out that such $a_i,b_i\in\Z$ are necessarily unique,
since $\lambda\in\R\setminus\Q$.
Hence we may then consider the point $x'=(x_1',\dots,x_n')$ given by $x_i'=b_i/\lambda$ if $i\in I$ and $x_i'=0$ otherwise, for all $i=1,\dots, n$.
First we notice that, since $\lambda x'$ is an integer point, we have
\begin{equation*}
\begin{split}
\G{n}\bigl(M_\lambda +(-1,1)^n\bigr)&=\G{n}\bigl((1-\lambda)K+\lambda(x-x')+\lambda x'+(-1,1)^n\bigr)\\
&=\G{n}\bigl((1-\lambda)K+\lambda(x-x')+(-1,1)^n\bigr).
\end{split}
\end{equation*}

Next, denoting by $x_0=x-x'$, we will show that there exists $\delta>0$ such that
\begin{equation}\label{e:lambda_irrat_mu}
\G{n}\bigl((1-\lambda)K+\lambda x_0+(-1,1)^n\bigr)\geq\G{n}\bigl((1-\mu)K+\mu x_0+(-1,1)^n\bigr)
\end{equation}
for all $\mu$ with $|\mu-\lambda|<\delta$.
Thus, taking a sequence $(r_m)_m\subset\Q\cap(0,1)$ with $\lim_{m\to\infty} r_m=\lambda$
and $|r_m-\lambda|<\delta$ for all $m\in\N$, we get, from the previous case, that
\begin{equation*}
\begin{split}
\G{n}\bigl((1-\lambda)K+\lambda x_0+(-1,1)^n\bigr)^{1/n}&\geq\G{n}\bigl((1-r_m)K+r_m x_0+(-1,1)^n\bigr)^{1/n}\\
&\geq(1-r_m)\G{n}(K)^{1/n}+r_m\G{n}(\{x_0\})^{1/n}\\
&\geq(1-r_m)\G{n}(K)^{1/n}
\end{split}
\end{equation*}
for all $m\in\N$, which yields the result.

To show \eqref{e:lambda_irrat_mu} we notice that, since $K$ is finite, $(1-\lambda)K+\lambda x_0+[-1,1]^n$ is a finite union of closed unit cubes and then, for any $\mu$, $(1-\mu)K+\mu x_0+[-1,1]^n$ is the union of the corresponding translates of the cubes that constitute $(1-\lambda)K+\lambda x_0+[-1,1]^n$.
Thus, there exists $\delta_1>0$ such that if
$z\in\Z^n$ satisfies that $z\notin(1-\lambda)K+\lambda x_0+[-1,1]^n$ then $z\notin(1-\mu)K+\mu x_0+(-1,1)^n$ for all $|\mu-\lambda|<\delta_1$.
Moreover, if $(1-\lambda)K+\lambda x_0+(-1,1)^n$ contains no integer boundary points, we may take $\delta=\delta_1$ and we are done.

So, we may assume that $\bd\bigl(\bigl((1-\lambda)K+\lambda x_0+(-1,1)^n\bigr)\cap\Z^n\bigr)\neq\emptyset$.
Let $z\in\Z^n$ be a boundary point of $(1-\lambda)K+\lambda x_0+(-1,1)^n$ and let $k\in K$.
On the one hand, if $z$ is in the boundary of the cube $(1-\lambda)k+\lambda x_0+(-1,1)^n$, there exist $i\in\{1,\dots,n\}$ and $\varepsilon\in\{-1,1\}$ such that $(1-\lambda)k_i+\lambda(x_i-x_i')+\varepsilon=z_i$, where $k_i, z_i\in\Z$ are the $i$-th components of $k,z$, respectively. This implies that $x_i-x_i'=k_i+(z_i-\varepsilon-k_i)/\lambda$ and then, from the definition of both $I$ and $x'$, we get that $i\in I$ and so $a_i=k_i+(z_i-\varepsilon-k_i)/\lambda$, which yields that $z_i-\varepsilon-k_i=0$ (because $\lambda$ is irrational) and thus $x_i-x_i'=a_i=k_i$. But then $(1-\mu)k_i+\mu(x_i-x_i')+\varepsilon=(1-\lambda)k_i+\lambda(x_i-x_i')+\varepsilon=z_i$ for all $\mu\in\R$,
and hence $z$ lies in the affine hull of a facet of the open cube $(1-\mu)k+\mu x_0+(-1,1)^n$, which implies that $z\notin(1-\mu)k+\mu x_0+(-1,1)^n$ for any $\mu$.
On the other hand, if $z$ is not in the boundary of the cube $(1-\lambda)k+\lambda x_0+(-1,1)^n$ (and so not contained in it either), there exists $\delta_{z,k}>0$ so that $z\notin(1-\mu)k+\mu x_0+(-1,1)^n$ for all $\mu$ with $|\mu-\lambda|<\delta_{z,k}$.

Since the number of integer points $z$ in the boundary of $(1-\lambda)K+\lambda x_0+(-1,1)^n$ is finite (and so is $K$), we may define $\delta_2=\min_{z,k}\delta_{z,k}$.
Altogether, \eqref{e:lambda_irrat_mu} holds for $\delta=\min\{\delta_1,\delta_2\}$.
\end{proof}

\smallskip

We conclude this section by showing that the classical Borell-Brascamp-Lieb inequality \eqref{e:BBL_means} can be obtained from the discrete version \eqref{e: BBL discreta (caso p-media)} under some mild assumptions on the functions involved:
\begin{theorem}
The discrete Borell-Brascamp-Lieb type inequality \eqref{e: BBL discreta (caso p-media)} implies the classical Borell-Brascamp-Lieb inequality \eqref{e:BBL_means}, provided that the functions $f,g$ are Riemann integrable and $h$ is upper semicontinuous.
\end{theorem}

Before proving this result, we notice that it is not possible to directly obtain any of the
discrete Brunn-Minkowski type inequalities stated so far in the paper
from the classical one \eqref{e:BM_additive}, by using the method of replacing the points by suitable
compact sets. As pointed out by Gardner and Gronchi in
\cite[pages~3996--3997]{GG}:

\medskip

\begin{center}
\begin{minipage}{0.9\textwidth}
``{\it it is worth remarking that the obvious idea of replacing the points
in the two finite sets by small congruent balls and applying the classical
Brunn-Minkowski inequality to the resulting compact sets is doomed to
failure. The fact that the sum of two congruent balls is a ball of twice
the radius introduces an extra factor of $1/2$ that renders the resulting
bound weaker than even the trivial bound (11) below}\,''.
\end{minipage}
\end{center}

\medskip

\noindent We clarify that (11) in \cite{GG} coincides with
\eqref{e:|A+B|>|A|+|B|-1} of the present paper.

\smallskip

In the following, for a function $\phi:\R^n\longrightarrow\R_{\geq0}$, we write $\phi^{\symbol_k}:\R^n\longrightarrow\R_{\geq0}$ to denote the function given by	
$\phi^{\symbol_k}(z) = \sup_{u\in(-2^{-k},2^{-k})^n}\phi(z+u)$ for all $z\in\R^n$.
\begin{proof}
Consider  $f,g,h:\R^n\longrightarrow\R_{\geq0}$  which satisfy the conditions of
Theorem \ref{t:BBL_means}, i.e., $f,g,h$ are measurable such that
\[
h((1-\lambda)x + \lambda y)\geq\Media{p}{f(x)}{g(y)}{\lambda}
\]
holds for all $x,y\in\R^n$ and for some fixed  $\lambda\in(0,1)$ and  $-1/n\leq p\leq\infty$.
Let $m\in \N$ and let $K = [-m,m]^n$. We will first show that
\begin{equation}\label{e:BBL_over_K}
\int_K h(x) \dlat x \geq \Media{\frac{p}{np+1}}{ \int_K f(x) \dlat x}{ \int_K g(x) \dlat x}{\lambda},
\end{equation}
for which we may assume (multiplying by $\chi_{_K}$ if necessary) that $f$, $g$ and $h$ vanish outside $K$.

For each $k\in\N$, we define the functions $f_k, g_k, h_k:\R^n\rightarrow\R_{\geq0}$ given by
\[ f_k(x) = \inf_{z \in x + [0,2^{-k}]^n} f(z), \quad g_k(x) = \inf_{z\in x + [0,2^{-k}]^n} g(z)\]
and
\[h_k(x) = \inf_{z\in x + [0,2^{-k}]^n} h(z). \]		
Writing for short $K_0=\inter K$, note that for any $x,y\in K_0$ we have
\begin{equation*}
\begin{split}
h_k((1-\lambda)x+\lambda y) &= \inf_{z \in (1-\lambda)x+\lambda y + [0,2^{-k}]^n} h(z)\\
&= \inf_{z \in (1-\lambda)(x+ [0,2^{-k}]^n)+\lambda (y+[0,2^{-k}]^n)} h(z)\\
&= \inf_{z_1 \in x+[0,2^{-k}]^n, z_2 \in y+[0,2^{-k}]^n} h((1-\lambda)z_1 + \lambda z_2)\\
&\geq \inf_{z_1 \in x+[0,2^{-k}]^n, z_2 \in y+[0,2^{-k}]^n} \Media{p}{f(z_1)}{g(z_2)}{\lambda}\\
&\geq \Media{p}{\inf_{z_1 \in x+[0,2^{-k}]^n} f(z_1)}{\inf_{z_2 \in y+[0,2^{-k}]^n} g(z_2)}{\lambda}\\
&= \Media{p}{f_k(x)}{g_k(y)}{\lambda},
\end{split}
\end{equation*}
and thus, we can use Theorem \ref{t: BBL discreta (caso p-media)} for $2^{-k}\Z^n$ to deduce that, for any $k\in \N$, we have
\begin{equation}\label{e:discBBL_lattice2-kZn}
2^{-kn}\sum_{z \in K\cap2^{-k}\Z^n} (h_k)^{\symbol_k}(z) \geq \Media{\frac{p}{np+1}}{2^{-kn}\sum_{x \in K_0\cap2^{-k}\Z^n} f_k(x)}{2^{-kn}\sum_{y \in K_0\cap2^{-k}\Z^n} g_k(y)}{\lambda},
\end{equation}
where, on the left-hand side, we have used that $K=[-m,m]^n$, thus
\[
\left(2^kK_0+(-1,1)^n\right)\cap \Z^n= 2^kK\cap \Z^n
\]
and
\[
\left(K_0+(-2^{-k},2^{-k})^n\right)\cap2^{-k}\Z^n= K\cap2^{-k}\Z^n.
\]
		
The level sets $\{x \in K: h(x)\geq t\}$ are closed, because $h$ is upper semicontinuous and $K$ is closed
(see \cite[Theorem~1.6]{RoWe}),
and then a standard straightforward computation shows that
\[\{x \in K: h(x)\geq t\}=\bigcap_{k=1}^\infty\left(\{x \in K: h(x)\geq t\}+(-2^{-k},2^{-k})^n\right).\]
Moreover, since $h$ vanishes outside $K$, we have
$\{x \in K: h(x)> t\}+(-2^{-k},2^{-k})^n\supset\{x \in K+[0,2^{-k}]^n: h^{\symbol_k}(x)> t\}$ for all $t>0$.
Thus, by Fubini's theorem and the monotone convergence theorem, we get
\begin{equation}\label{e: ints_h_h^symbol}
\begin{split}
\int_K h(x) \dlat x &= \int_0^\infty \vol\bigl(\{x \in K: h(x)\geq t\}\bigr) \dlat t\\
&= \int_0^\infty \vol\left(\bigcap_{k=1}^\infty\left(\{x \in K: h(x)\geq t\}+(-2^{-k},2^{-k})^n\right)\right) \dlat t\\
&= \int_0^\infty \lim_{k\rightarrow\infty} \vol\left(\{x \in K: h(x)\geq t\}+(-2^{-k},2^{-k})^n\right) \dlat t \\
&= \lim_{k\rightarrow\infty} \int_0^\infty \vol\left(\{x \in K: h(x)\geq t\}+(-2^{-k},2^{-k})^n\right) \dlat t \\
&\geq \lim_{k\rightarrow\infty}\int_0^\infty \vol\bigl(\{x \in K+[0,2^{-k}]^n: h^{\symbol_k}(x)> t\}\bigr) \dlat t \\
&=\lim_{k\rightarrow\infty}\int_{K+[0,2^{-k}]^n} h^{\symbol_k}(x) \dlat x.
\end{split}
\end{equation}


Now we show that, given $z\in\R^n$, $h^{\symbol_k}(x)\geq(h_k)^{\symbol_k}(z)$ for all $x\in z+[0,2^{-k}]^n$. Indeed, we have
\begin{equation*}
\begin{split}
h^{\symbol_k}(x)&=\sup_{u\in(-2^{-k},2^{-k})^n} h(x+u)\geq\sup_{u\in(-2^{-k},2^{-k})^n} \inf_{v\in[0,2^{-k}]^n}h(z+v+u)\\
&=\sup_{u\in(-2^{-k},2^{-k})^n} h_k(z+u)=(h_k)^{\symbol_k}(z).
\end{split}
\end{equation*}
This, together with \eqref{e: ints_h_h^symbol} and the fact that $K+[0,2^{-k}]^n=K\cap2^{-k}\Z^n + [0,2^{-k}]^n$, implies that
\begin{equation*}
\int_K h(x) \dlat x \geq \lim_{k\rightarrow\infty}\int_{K+[0,2^{-k}]^n} h^{\symbol_k}(x) \dlat x
\geq \lim_{k\rightarrow\infty} 2^{-kn} \sum_{z \in K\cap2^{-k}\Z^n} (h_k)^{\symbol_k}(z).
\end{equation*}
Furthermore, since $f$ is Riemann integrable and $2^{-kn}\sum_{x \in K_0\cap2^{-k}\Z^n} f_k(x)$ is a \textit{lower sum} of $f\cdot\chi_{_{(-m,m]^n}}$ for the partition $\{x + [0,2^{-k}]^n\subset K: x\in2^{-k}\Z^n\}$ of $K$,
it is clear that
\begin{equation*}
\lim_{k\rightarrow\infty} 2^{-kn}\sum_{x \in K_0\cap2^{-k}\Z^n} f_k(x) = \int_K f(x)\dlat x.
\end{equation*}

Here we observe that it was crucial to work with $K_0$ in order to get a lower sum of $f\cdot\chi_{_{(-m,m]^n}}$
for the above partition. We also point out the necessity of considering the characteristic function $\chi_{_{(-m,m]^n}}$ instead of $\chi_{_{[-m,m]^n}}$, which has no influence when computing the above integral: in this way, the function $f\cdot\chi_{_{(-m,m]^n}}$ vanishes on the points of the corresponding facets of the cube.

The same holds for the function $g$ and then, taking limits on both sides of \eqref{e:discBBL_lattice2-kZn},
we get \eqref{e:BBL_over_K}.
Since \eqref{e:BBL_over_K} is true for $K=[-m,m]^n$, for every $m\in\N$, the proof is now concluded because
\[ \int_{\R^n} \phi(x) \dlat x = \lim_{m\rightarrow\infty} \int_{[-m,m]^n} \phi(x)\dlat x,\]
for every measurable function $\phi:\R^n\rightarrow\R_{\geq0}$.
\end{proof}

It is well-known that a function is Riemann integrable if and only if it is continuous almost everywhere. Since the boundary of a convex set has null measure (and from the characterization of the upper semicontinuity in terms of the level sets) we get the following result, as a straightforward consequence of the previous one.
\begin{corollary}
The discrete Brunn-Minkowski type inequality \eqref{e: BM_lattice_point_no_G(K)G(L)>0} implies the classical Brunn-Minkowski inequality \eqref{e:BM} for bounded convex sets $K$ and $L$.
\end{corollary}
We notice the necessity of assuming convexity in the latter result: for any measurable sets $K, L \subset \R^n$
of positive volume, containing no rational point, one cannot expect to recover the Brunn-Minkowski inequality
\eqref{e:BM} with the above method of shrinking the lattice $\Z^n$, since $K, L$ have no point in $2^{-k}\Z^n$, for any $k\in\N$.

\section{Relations with other inequalities}

The multiplicative version of the Brunn-Minkowski inequality is, among all its equivalent forms, the one that is
naturally connected to the Pr\'ekopa-Leindler inequality (the case $p=0$ of Theorem \ref{t:BBL_means}).
It asserts that if $\lambda\in(0,1)$ and $K$ and $L$ are non-empty compact subsets of $\R^n$ then
\begin{equation}\label{e:BM_multip}
\vol\bigl((1-\lambda)K+\lambda L\bigr)\geq
\vol(K)^{1-\lambda}\vol(L)^{\lambda}.
\end{equation}

In the discrete setting, considering now non-empty bounded sets $K, L\subset\R^n$,
from \eqref{e: BM_lattice_point_general} for $p=0$ we get
\begin{equation}\label{e:disc_multip_BM}
\G{n}\bigl((1-\lambda)K+\lambda L+(-1,1)^n\bigr)\geq\G{n}(K)^{1-\lambda}\G{n}(L)^\lambda.
\end{equation}
Regarding other possible discrete versions of \eqref{e:BM_multip}, we have the following engaging and elegant result, shown very recently by Halikias, Klartag and Slomka in \cite{HKS} (see also \cite{KL}):
\begin{thm}[\cite{HKS}]\label{t:HKS}
Let $\lambda\in(0,1)$ and let $f,g,h,k:\Z^n\longrightarrow\R_{\geq0}$ be functions
such that
\[h(\floor{(1-\lambda)x+\lambda y})k(\ceil{\lambda x+(1-\lambda) y})\geq f(x)g(y)\]
for all $x,y\in\Z^n$, where $\floor{x}=(\floor{x_1},\dots,\floor{x_n})$ and
$\ceil{x}=(\ceil{x_1},\dots,\ceil{x_n})$. Then
\[\left(\sum_{x\in\Z^n}h(x)\right)\left(\sum_{x\in\Z^n}k(x)\right)\geq
\left(\sum_{x\in\Z^n}f(x)\right)\left(\sum_{x\in\Z^n}g(x)\right).\]
\end{thm}
As they observed, when applying the above result to the functions $f=\chi_{_K}$, $g=\chi_{_L}$, $h=\chi_{_{\frac{K+L}{2}+(-1,0]^n}}$
and $k=\chi_{_{\frac{K+L}{2}+[0,1)^n}}$, one has
\begin{equation*}
\G{n}\left(\frac{K+L}{2}+(-1,0]^n\right)\G{n}\left(\frac{K+L}{2}+[0,1)^n\right)\geq\G{n}(K)\G{n}(L),
\end{equation*}
which yields the discrete multiplicative Brunn-Minkowski type inequality
\begin{equation}\label{e:disc_multip_BM_HKS}
\G{n}\left(\frac{K+L}{2}+[0,1]^n\right)\geq\sqrt{\G{n}(K)\G{n}(L)}.
\end{equation}
We notice that the sole difference between \eqref{e:disc_multip_BM_HKS} and \eqref{e:disc_multip_BM} (for $\lambda=1/2$) is the necessity of adding either the closed cube of edge length $1$ or the open cube of edge length $2$, respectively. However, they are not comparable. Indeed, let $n=1$ and let $K=L=[-x,x]$ with $x\in\R_{\geq0}$. On the one hand,
for $x\in\Z$, we have $\G{1}\bigl((K+L)/2+[0,1]\bigr)=2x+2>2x+1=\G{1}\bigl((K+L)/2+(-1,1)\bigr)$.
On the other hand, for $x\notin\Z$, we get
$\G{1}\bigl((K+L)/2+[0,1]\bigr)=2\floor{x}+2<2\floor{x}+3=\G{1}\bigl((K+L)/2+(-1,1)\bigr)$.

\smallskip

As pointed out in Remark \ref{r:best_interval}, inequality \eqref{e: BM_lattice_point_no_G(K)G(L)>0} is in general not true (even for $\lambda=1/2$) by just adding the cube $[0,1)^n$ to the convex combination $(K+L)/2$.
However, this can be solved by just considering the closed cube $[0,1]^n$, i.e., we show that
inequality \eqref{e:disc_multip_BM_HKS} also admits a $(1/n)$-form:
\begin{theorem}\label{t: BM_lattice_point_1/2}
Let $K, L \subset \R^n$ be bounded sets such that $\G{n}(K)\G{n}(L)>0$. Then
\begin{equation}\label{e: BM_lattice_point_1/2}
\G{n}\left(\frac{K + L}{2} + [0,1]^n\right)^{1/n}\geq \frac{\G{n}(K)^{1/n}+\G{n}(L)^{1/n}}{2}.
\end{equation}
The inequality is sharp.
\end{theorem}

As in Theorem \ref{t: BM_lattice_point_[-(q-1)/q,(q-1)/q]}, we just have to prove the corresponding one-dimen\-sio\-nal case, collected in Lemma \ref{l: G(M+[0,1])>=1/2(G(K)+G(L))}.
Then, the proof of Theorem \ref{t: BM_lattice_point_1/2} is completed in a way similar to the proof of Theorem
\ref{t: BBL discreta (caso p-media)} with the particular functions $f=\chi_{_K}$, $g=\chi_{_L}$ and $h=\chi_{_{\frac{K+L}{2}}}$, and $\lambda=1/2$, replacing there the use of inequality \eqref{e: 1-dim_BM_G} by
\eqref{e: G((K+L)/2+[0,1])>=1/2(G(K)+G(L))}.

\smallskip

Finally, to show that equality may be attained, we consider $K=-L=[0,m]^n$ with $m\in\N$ odd, for which
we have $\G{n}\bigl((K + L)/2+[0,1]^n\bigr)=\G{n}(K)=\G{n}(L)=(m+1)^n$.

\begin{lemma}\label{l: G(M+[0,1])>=1/2(G(K)+G(L))}
Let $K, L \subset \R$ be non-empty bounded sets. Then
\begin{equation}\label{e: G((K+L)/2+[0,1])>=1/2(G(K)+G(L))}
\G{1}\left(\frac{K+L}{2}+[0,1]\right) \geq \frac{\G{1}(K) + \G{1}(L)}{2}.
\end{equation}
\end{lemma}	
\begin{proof}
The proof is completely analogous to that of Lemma \ref{l: G(alpha K+beta L+[-(q-1)/q,(q-1)/q])}, and thus
we include here just the slight differences.

First we notice that, for any $x,y\in\R$, we have
\begin{equation}\label{e:property_floor_(x+y)/2}
\floor{\frac{x+y}{2}}+\frac{1}{2}\geq\frac{\floor{x}+\floor{y}}{2}.
\end{equation}
Indeed, if $\floor{x}+\floor{y}$ is even then $\floor{(x+y)/2}=(\floor{x}+\floor{y})/2$ whereas if
$\floor{x}+\floor{y}$ is odd we get $\floor{(x+y)/2}\geq(\floor{x}+\floor{y}-1)/2$.

\smallskip

Now, we notice that if $M\supset(K+L)/2$ is such that
$M +[0,1]= [a_1, b_1]$ is a (non-empty) compact interval, we have on the one hand that
$\G{1}\bigl(M+[0,1]\bigr) = \floor{b_1} - \ceil{a_1} +1$.
Moreover, denoting by $a = \inf K$, $b = \sup K$, $c = \inf L$ and $d = \sup L$, we clearly get
$\G{1}(K)\leq \G{1}([a,b]) = \floor{b} - \ceil{a} +1$ and $\G{1}(L) \leq \G{1}([c,d]) = \floor{d} - \ceil{c} +1$.
On the other hand, the inclusion $(K + L)/2 \subset M$ implies that
\begin{equation*}
a_1\leq\frac{a+c}{2}\leq\frac{b+d}{2}\leq b_1-1.
\end{equation*}
Altogether, and using \eqref{e:property_floor_(x+y)/2} jointly with the fact that
$\ceil{x}=-\floor{-x}$ for any $x\in\R$, we obtain
\begin{equation*}
\begin{split}
\G{1}\bigl(M+[0,1]\bigr) &= \floor{b_1} -\ceil{a_1} + 1 \geq  \floor{\frac{b+d}{2}} - \ceil{\frac{a+c}{2}} + 2\\
&\geq\frac{\floor{b}+\floor{d}}{2}-\frac{\ceil{a}+\ceil{c}}{2}+1=\frac{\G{1}([a,b])+\G{1}([c,d])}{2}\\
&\geq\frac{\G{1}(K)+\G{1}(L)}{2}.
\end{split}
\end{equation*}
The proof is then completed in a way analogous to the proof of Lemma \ref{l: G(alpha K+beta L+[-(q-1)/q,(q-1)/q])}.
\end{proof}

We would like to point out that the corresponding version of Theorem \ref{t: BM_lattice_point_1/2} for an arbitrary $\lambda\in(0,1)$ is, in general, not true. Indeed, taking $K=[0,1]^n$, $L=[-5,6]^n$ and $\lambda=1/3$ one has that
\begin{equation*}
\begin{split}
\G{n}\bigl((1-\lambda)K+\lambda L+[0,1]^n\bigr)^{1/n}
&=\G{n}\left(\left[\frac{-5}{3},\frac{11}{3}\right]^n\right)^{1/n}=5\\
&<\frac{16}{3}=(1-\lambda)\G{n}(K)^{1/n}+\lambda\G{n}(L)^{1/n}.
\end{split}
\end{equation*}

\smallskip

We include here an open question that arose during our study:
\begin{question}\label{q:G{n}(K)G{n}(L)>0}
Regarding the statement of Theorem \ref{t: BM_lattice_point_1/2},
is the assumption $\G{n}(K)\G{n}(L)>0$ necessary?
\end{question}

We conclude the section by proving an inequality similar to \eqref{e:B-M_discrete_adding} but in the spirit of \eqref{e: BM_lattice_point_no_G(K)G(L)>0}, namely that one may add another (fixed) set to the Minkowski sum $A+B$ instead of considering the extension $\bar{A}$ of $A$. We show that an appropriate set to be taken into account in this respect is the lattice cube $\{0,1\}^n$, which also fits well with inequality \eqref{e: BM_lattice_point_1/2}.
\begin{theorem}\label{t:B-M_discrete_sum}
Let $A,B\subset\Z^n$ be finite, $A,B\neq\emptyset$. Then
\begin{equation}\label{e:B-M_discrete_sum}
\bigl|A+B+\{0,1\}^n\bigr|^{1/n}\geq |A|^{1/n}+|B|^{1/n}.
\end{equation}
The inequality is sharp.	
\end{theorem}
The idea of the proof we present here goes back to \cite[Lemma~2.4]{GT}.
\begin{proof}
Taking into account the relation
\[A+B+\{0,1\}^n+[0,1]^n=A+[0,1]^n+B+[0,1]^n,\]
jointly with the fact that $\vol(X+[0,1]^n)=|X|$, for any non-empty and finite set $X\subset\Z^n$,
the result directly follows from the Brunn-Minkowski inequality \eqref{e:BM_additive}. Finally, to show that the inequality is sharp, we consider the lattice cubes $A=\{0,\dots,m_1\}^n$ and $B=\{0,\dots,m_2\}^n$, for $m_1, m_2\in\Z_{\geq0}$.
\end{proof}

We conclude the paper by comparing inequalities \eqref{e:B-M_discrete_adding} and \eqref{e:B-M_discrete_sum}; again, we refer the reader to Section 2.1 in \cite{HCIYN} for the precise definition of the extension $\bar{A}$ of a given non-empty finite set $A\subset\Z^n$.

Since $A+(\{0\}^{n-1}\times\{0,1\})$ contains at least the same amount of points that the union of $A$ and its maximal cardinality section, and taking into account that $\{0,1\}^n=\bigl(\{0,1\}^{n-1}\times\{0\}\bigr)+\bigl(\{0\}^{n-1}\times\{0,1\}\bigr)$,
from the definition of $\bar{A}$ is then immediate that $\bigl|\bar{A}\bigr|\leq\bigl|A+\{0,1\}^n\bigr|$.
However, inequalities \eqref{e:B-M_discrete_adding} and \eqref{e:B-M_discrete_sum} are not comparable.
Indeed, if we consider on the one hand $A=\{(0,0), (1,0), (2,0), (1,1)\}$ and $B=\{0,1\}^2$, we have
$\bigl|\bar{A}\bigr|=10=\bigl|A+\{0,1\}^2\bigr|$ (with $\bar{A}\neq A+\{0,1\}^2$) and $\bigl|\bar{A}+B\bigr|=20 > 18 = \bigl|A+\{0,1\}^2+B\bigr|$ (see Figure \ref{fig: bar{A}+B vs A+C+B}).
On the other hand, for $A=\{(0,0), (0,1), (1,1), (4,1)\}$ and $B=\{0,1\}^2$ we obtain
$\bigl|\bar{A}+B\bigr|=21 < 24 = \bigl|A+\{0,1\}^2+B\bigr|$.

\smallskip

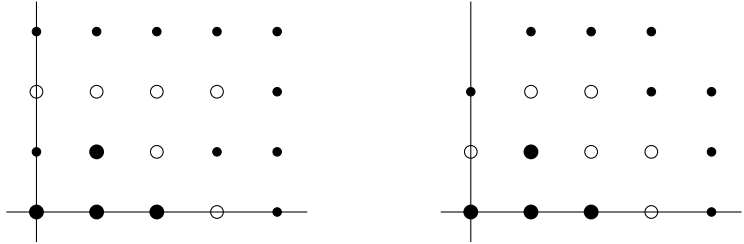
\begin{figure}[h]\label{fig: bar{A}+B vs A+C+B}
\centering
\begin{tikzpicture}[A/.style ={thick,color=black}, Amp/.style ={fill=none}, AmppB/.style ={color=black}, scale=0.8]
\draw (-0.5,0) -- (4.5,0);
\draw (0,-0.5) -- (0,3.5);
\filldraw[A] (0,0) circle (3pt);
\filldraw[A] (1,1) circle (3pt);
\filldraw[A] (1,0) circle (3pt);
\filldraw[A] (2,0) circle (3pt);
\filldraw[Amp] (0,2) circle (3pt);
\filldraw[Amp] (1,2) circle (3pt);
\filldraw[Amp] (2,2) circle (3pt);
\filldraw[Amp] (2,1) circle (3pt);
\filldraw[Amp] (3,2) circle (3pt);
\filldraw[Amp] (3,0) circle (3pt);
\filldraw[AmppB] (0,1) circle (2pt);
\filldraw[AmppB] (0,3) circle (2pt);
\filldraw[AmppB] (1,3) circle (2pt);
\filldraw[AmppB] (2,3) circle (2pt);
\filldraw[AmppB] (3,3) circle (2pt);
\filldraw[AmppB] (4,3) circle (2pt);
\filldraw[AmppB] (4,2) circle (2pt);
\filldraw[AmppB] (3,1) circle (2pt);
\filldraw[AmppB] (4,1) circle (2pt);
\filldraw[AmppB] (4,0) circle (2pt);
\end{tikzpicture}
\hspace{1.5cm}
\begin{tikzpicture}[A/.style ={thick,color=black}, Amp/.style ={fill=none}, AmppB/.style ={color=black}, scale=0.8]
\draw (-0.5,0) -- (4.5,0);
\draw (0,-0.5) -- (0,3.5);
\filldraw[A] (0,0) circle (3pt);
\filldraw[A] (1,1) circle (3pt);
\filldraw[A] (1,0) circle (3pt);
\filldraw[A] (2,0) circle (3pt);
\filldraw[Amp] (0,1) circle (3pt);
\filldraw[Amp] (1,2) circle (3pt);
\filldraw[Amp] (2,2) circle (3pt);
\filldraw[Amp] (2,1) circle (3pt);
\filldraw[Amp] (3,1) circle (3pt);
\filldraw[Amp] (3,0) circle (3pt);
\filldraw[AmppB] (0,2) circle (2pt);
\filldraw[AmppB] (1,3) circle (2pt);
\filldraw[AmppB] (2,3) circle (2pt);
\filldraw[AmppB] (3,3) circle (2pt);
\filldraw[AmppB] (4,2) circle (2pt);
\filldraw[AmppB] (3,2) circle (2pt);
\filldraw[AmppB] (4,1) circle (2pt);
\filldraw[AmppB] (4,0) circle (2pt);
\end{tikzpicture}
\caption{Left: the sets $A$ (thick points), $\bar{A}$ (thick and hollow points) and $\bar{A}+B$. Right:
the sets $A$ (thick points), $A+\{0,1\}^2$ (thick and hollow points) and $A+\{0,1\}^2+B$.}
\end{figure}

\medskip

\noindent {\it Acknowledgements.} We are very grateful to the anonymous referee for
her/his very helpful comments and remarks which have allowed us to improve the presentation of this work; we specially thank the referee for bringing to our attention the simpler and more elegant proof (than the one we originally did) of
Theorem \ref{t:B-M_discrete_sum}.
We would like to thank M. A. Hern\'andez Cifre for her very valuable suggestions and the careful reading of this paper.
We also thank D. Halikias, B. Klartag and B. A. Slomka for bringing to our attention Theorem \ref{t:HKS}
and J. Li for pointing out the possibility of considering the statement of Theorem \ref{t:B-M_discrete_sum}.


\begin{thebibliography}{99}

\bibitem{Brt} Barthe, F. ``Autour de l'in\'egalit\'e de Brunn-Minkowski.''
{\it Ann. Fac. Sci. Toulouse Math. (6)} 12, no. 2 (2003): 127--178.

\bibitem{Borell} Borell, C. ``Convex set functions in $d$-space.'' {\it Period.
Math. Hungar.} 6 (1975): 111--136.

\bibitem{BL} Brascamp, H. J. and Lieb, E. H. ``On extensions of the
Brunn-Minkowski and Pr\'ekopa-Leindler theorems, including inequalities
for log concave functions and with an application to the diffusion
equation.'' {\it  J. Func. Anal.} 22, no. 4 (1976): 366--389.

\bibitem{Bu} Bullen, P. S. {\it Handbook of means and their
inequalities}, Mathematics and its Applications vol. 560, Revised from the
1988 original, Kluwer Academic Publishers Group, Dordrecht, 2003.

\bibitem{G} Gardner, R. J. ``The Brunn-Minkowski inequality.'' {\it Bull. Amer. Math. Soc.}
39, no. 3 (2002): 355--405.

\bibitem{GG} Gardner, R. J. and Gronchi, P. ``A Brunn-Minkowski inequality for the integer lattice.''
{\it Trans. Amer. Math. Soc.} 353, no. 10 (2001): 3995--4024.

\bibitem{GT} Green, B. and Tao, T. ``Compressions, convex geometry and the Freiman-Bilu theorem.''
{\it Q. J. Math.} 57, no. 4 (2006): 495--504.

\bibitem{HKS} Halikias, D., Klartag, B. and Slomka, B. A. ``Discrete variants of Brunn-Minkowski type inequalities.''
{\it Submitted}, \href{https://arxiv.org/abs/1911.04392}{arXiv:1911.04392}.

\bibitem{HaLiPo} Hardy, G. H., Littlewood, J. E. and P{\'o}lya, G. {\it Inequalities}.
Cambridge Mathematical Library, Reprint of the 1952 edition. Cambridge: Cambridge University Press, 1988.

\bibitem{HCIYN} Hern\'andez Cifre, M. A., Iglesias, D. and Yepes Nicol\'as, J.
``On a discrete Brunn-Minkowski type inequality.'' {\it SIAM J. Discrete Math.}
32 (2018): 1840--1856.

\bibitem{IYN} Iglesias, D. and Yepes Nicol\'as, J. ``On discrete Borell-Brascamp-Lieb inequalities.''
To appear in {\it Rev. Matem\'atica Iberoamericana}.

\bibitem{KL} Klartag, B. and Lehec, J.
``Poisson processes and a log-concave Bernstein theorem.'' {\it Stud. Math.}
247, no. 1, (2019): 85--107.



\bibitem{RoWe} Rockafellar, R. T. and Wets, R. J.-B. {\it Variational
analysis}. Grundlehren der Mathematischen Wissenschaften [Fundamental
Principles of Mathematical Sciences], 317. Berlin: Springer-Verlag, 1998.

\bibitem{Ru1} Ruzsa, I. Z. ``Sum of sets in several dimensions.''
{\it Combinatorica} 14 (1994): 485--490.

\bibitem{Ru2} Ruzsa, I. Z. ``Sets of sums and commutative graphs.''
{\it Studia Sci. Math. Hungar.} 30 (1995): 127--148.

\bibitem{Sch2} Schneider, R. {\it Convex bodies: The Brunn-Minkowski
theory}, 2nd expanded ed. Encyclopedia of Mathematics and its
Applications, 151. Cambridge: Cambridge University Press, 2014.

\bibitem{Tao} Tao, T. and  Vu, V. {\it Additive combinatorics}, Cambridge
Studies in Advanced Mathematics, 105. Cambridge: Cambridge University Press, 2006.

\end{thebibliography}
\end{document}